\documentclass[11pt]{amsart}

\usepackage{amssymb}
\usepackage{amsmath}
\usepackage{amsthm}
\usepackage{amsrefs}
\usepackage{enumerate}
\usepackage{comment}
\usepackage{graphicx}
\usepackage{hyperref}

\usepackage{xcolor}

\theoremstyle{plain}

\newtheorem{theorem}{Theorem}[section]

\newtheorem{lemma}[theorem]{Lemma}
\newtheorem{proposition}[theorem]{Proposition}
\newtheorem{problem}[theorem]{Problem}

\newtheorem{fact}[theorem]{Fact}
\newcounter{maintheorem}

\newtheorem{mainth}[maintheorem]{Theorem}

\theoremstyle{definition}

\newtheorem{definition}[theorem]{Definition}
\newtheorem{remark}[theorem]{Remark}
\newtheorem{example}[theorem]{Example}

\renewcommand{\leq}{\leqslant}
\renewcommand{\geq}{\geqslant}

\newcommand{\R}{\mathbb{R}}
\newcommand{\N}{\mathbb{N}}

\newcommand{\inte}{\mathrm{int}\:}

\renewcommand{\epsilon}{\varepsilon}
\renewcommand{\phi}{\varphi}

\hypersetup{
    pdftitle={},
    pdfauthor={},
    pdfmenubar=false,
    pdffitwindow=true,
    pdfstartview=FitH,
    colorlinks=true,
    linkcolor=blue,
    citecolor=magenta,
    urlcolor=cyan
}

\author[C.A.~De~Bernardi]{Carlo Alberto De Bernardi}

\address{Dipartimento di Matematica per le Scienze economiche, finanziarie ed attuariali, Universit\`a Cattolica del Sacro Cuore, 20123 Milano,Italy}

\email{carloalberto.debernardi@unicatt.it}
\email{carloalberto.debernardi@gmail.com}

\author[A.~Preti]{Alessandro Preti}
\address{Politecnico di Milano, Dipartimento di Matematica, Piazza Leonardo da Vinci 32, 20133 Milano, Italy.}
\email{alessandro3.preti@mail.polimi.it}
\email{alessandro.preti@outlook.it}

\author[J.~Somaglia]{Jacopo Somaglia}
\address{Politecnico di Milano, Dipartimento di Matematica, Piazza Leonardo da Vinci 32, 20133 Milano, Italy.}
\email{jacopo.somaglia@polimi.it}



\thanks{
The research of the first author was supported by the INdAM -
GNAMPA Project,  CUP E53C23001670001, and by MICINN project PID2020-112491GB-I00 (Spain).
The research of the third author was supported by the INdAM -
GNAMPA Project,  CUP E53C23001670001.
}
\title[Smooth rotund norms which are  not MLUR]{A note on smooth rotund norms which are  not midpoint locally uniformly rotund}

\begin{document}
	
	\begin{abstract} 
		We prove that every separable infinite-dimensional Banach space admits a G\^ateaux smooth and rotund norm which is not midpoint locally uniformly rotund. Moreover, by using a similar technique, we provide in every infinite-dimensional Banach space with separable dual a Fr\'echet  smooth and weakly uniformly rotund norm which is not midpoint locally uniformly rotund. These two results provide a positive answer to some open problems by A. J. Guirao, V. Montesinos, and V. Zizler.
	\end{abstract}

	\maketitle
\section{Introduction}
In \cite{DEBESOM-ALUR} the first and the third named authors proved the existence, in every infinite-dimensional separable Banach space, of a rotund G\^ateaux smooth norm which satisfies the Kadec property but it is not locally uniformly rotund (LUR). In particular, this result solves  in the affirmative a problem posed in the recent monograph \cite{GMZ} by A. J. Guirao, V. Montesinos, and V. Zizler (\cite[Problem 52.1.1.5]{GMZ}).
The same monograph contains many other problems in the same direction, basically devoted to the study of  relations between the different rotundity or smoothness properties of a norm, and few of them have recently been solved. We mention some very recent results in this direction (we refer to Definition \ref{d: basic def} and \cite{GMZ} for unexplained definitions of different notions of rotundity).
\begin{enumerate}
    \item In \cite{HQ} P. H\'ajek and A. Quilis have constructed in every infinite-dimensional separable Banach space a norm which is LUR but not URED (uniformly rotund in every direction), solving positively \cite[Problem 52.3.4]{GMZ}.
    \item 
In \cite{Q}  A. Quilis 
 has constructed in every separable Banach space admitting a $C^k$-smooth norm an equivalent norm which is $C^k$-smooth but fails to be uniformly Gâteaux in any direction. Moreover, in the same paper, he proved that,  for any infinite $\Gamma$, $c_0(\Gamma)$ admits a $C^\infty$-smooth norm whose ball is dentable but whose sphere lacks any extreme points. 
In particular, these results solve affirmatively \cite[Problem 52.1.2.1]{GMZ} and \cite[Problem 52.1.3.3]{GMZ}.
\end{enumerate}

 The aim of our paper is to study the following  two problems.

\begin{problem}
\label{p: RnotMLUR}
    Does every infinite-dimensional separable Banach space admit an equivalent rotund and  G\^ateaux smooth norm which is not MLUR?
\end{problem}

\begin{problem}
\label{p: WURnotMLUR}
    Does every infinite-dimensional Banach space with separable dual admit an equivalent WUR and Fr\'echet smooth norm which is not MLUR? 
\end{problem}

Let us notice that a solution in the affirmative to Problems~\ref{p: RnotMLUR}~and~\ref{p: WURnotMLUR} also provides a solution (in the affirmative) to  
{\cite[Problem~52.3.3]{GMZ}} and 
{\cite[Problem~52.3.6]{GMZ}}, in which smoothness assumptions are not required. 

Let us describe the structure of the paper and its main results. In Section~\ref{sec: notation}, after some notations and preliminaries,  we recall a construction by S.~Draga contained in \cite{D15}, which was introduced for proving the existence of an equivalent WLUR norm that is not LUR, in any infinite-dimensional Banach space with  separable dual. We observe that such a norm also provides a solution in the affirmative to \cite[Problem~52.3.3]{GMZ}. Moreover, we observe that a similar construction also solves affirmatively   
\cite[Problem~52.3.6]{GMZ}. However, it is 
 standard to show that the  norm provided in \cite{D15} is not G\^ateaux smooth and hence such a renorming technique cannot be directly used to solve   Problems~\ref{p: RnotMLUR}~and~\ref{p: WURnotMLUR} (see Section~\ref{sec: draganorm} for the details). The main result of the paper consists in  an alternative construction that solves 
these two problems in the affirmative. More precisely, in Section~\ref{sec: main result}, we prove that (see Theorems~\ref{th: mainG}~and~\ref{th: mainF}): 

\begin{mainth}\label{th: thm A}
Let $(X,\|\cdot\|)$ be a separable Banach space [a Banach space with separable dual, respectively], then there exists a
 rotund G\^ateaux  smooth [a weak uniformly rotund and Fr\'echet smooth, respectively]  equivalent norm  which is not midpoint locally uniformly rotund. Moreover, such a norm can be chosen arbitrarily close to $\|\cdot\|$.
\end{mainth} 

\noindent Moreover, in Remark~\ref{r: URED} we observe that our construction, also  provide, in every separable space, an equivalent norm that is  G\^ateaux  smooth and URED but not LUR. In particular, this solves in the affirmative  \cite[Problem 52.3.5]{GMZ} (in which the smoothness assumption is not required). 

Finally, we conclude the section by showing that every separable Banach space admits an equivalent rotund norm in which the set of MLUR points is porous (and hence has empty interior) in the unit sphere. On the other hand seems to be open the following problem: can every infinite-dimensional separable space be renormed to be rotund and  G\^ateaux  smooth  but with no LUR point in the unit sphere?

In the Appendix we provide an alternative proof of Lemma \ref{lemma: flatnormingF}.

\section{Preliminaries}\label{sec: notation}

We follow the notations and terminologies introduced in \cite{DEBESOM-ALUR}. Throughout this  paper, Banach spaces are real and infinite-dimensional. Let $X$ be a normed space, by $X^*$ we denote the dual space of $X$. By $B_X$, $U_X$, and $S_X$ we denote the closed unit ball, the open unit ball, and the unit sphere of $X$, respectively.
	Moreover, in situations when more than one
	norm on $X$ is considered, we denote by $B_{(X,\|\cdot\|)}$, $U_{(X,\|\cdot\|)}$ and $S_{(X,\|\cdot\|)}$ the closed unit ball, the open unit ball, and the closed unit sphere  with respect to the norm ${\|\cdot\|}$, respectively. By $\|\cdot\|^*$ we denote the dual norm of $\|\cdot\|$.
 A \emph{biorthogonal system} in a separable Banach space $X$ is a system $(e_n, f_n)_{n\in\N}\subset X\times X^*$, such that $f_n(e_m)=\delta_{n,m}$ ($n,m\in \N$). A biorthogonal system is \emph{fundamental} if ${\rm span}\{e_n\}_{n\in\N}$ is dense in $X$; it is \emph{total} when ${\rm span}\{f_n\} _{n\in\N}$ is $w^*$-dense in $X^*$. A \emph{Markushevich basis} (M-basis) is a fundamental and total biorthogonal system. An M-basis $(e_n,f_n)_{n\in \N}$ is called \textit{bounded} if $\sup_{n\in\N}\|e_{n}\|\|f_n\|^*$ is finite. By \cite[Theorem 1.27]{HMVZ} every separable Banach space admits a bounded M-basis. We refer to \cite{HMVZ,HRST,RS23} and references therein for some more information on M-bases.

Given a subset $K$ of $X$, a {\em slice} of $K$ is a set of the form $$S(K,x^*,\alpha) := \{x\in K;\,  x^*(x)> \sup x^*(K)-\alpha\},$$ where $\alpha> 0$ and $x^*\in X^*\setminus\{0\}$ is bounded above on $K$.
For convenience of the reader, we recall some standard notions which we list in the following definition.

\begin{definition}\label{d: basic def}
Let $x\in S_X$ and $x^*\in S_{X^*}$. We say that:
\begin{itemize}
\item  $x$ is  an {\em extreme point} of $B_X$ if $x=y=z$, whenever $y,z\in B_X$ and $2x=y+z$; 
\item $X$ is \textit{rotund} if each point of $S_X$ is an extreme point of $B_X$;
\item $x$ is a {\em locally uniformly rotund (LUR) point} of $B_X$	if  $x_n\to x$, whenever $\{x_n\}_{n\in\N}\subset X$ is such that $$\lim_n 2\|x\|^2 + 2\|x_n\|^2 - \|x+x_n\|^2=0;$$
\item $x$ is a \textit{midpoint locally uniformly rotund (MLUR) point} if whenever $\{x_n\}_{n\in\N},$ $ \{y_n\}_{n\in\N} \subset X$ are such that $\|x_n\|\to \|x\|$, $\|y_n\|\to \|x\|$ and $\|\frac{x_n+y_n}{2}-x\|\to 0$, then $\|x_n-y_n\|\to 0$;
\item $x$ is a {\em weakly locally uniformly rotund (WLUR) point} of $B_X$	if $x_n\to x$ weakly, whenever $\{x_n\}_{n\in\N}\subset X$ is such that $$\lim_n 2\|x\|^2 + 2\|x_n\|^2 - \|x+x_n\|^2=0;$$
\item $x$ is a {\em weakly midpoint locally uniformly rotund (WMLUR) point} of $B_X$	if $x_n\to x$ weakly, whenever $\{x_n\}_{n\in\N}, \{y_n\}_{n\in\N} \subset X$ are such that $\|x_n\|\to \|x\|$, $\|y_n\|\to \|x\|$ and $\|\frac{x_n+y_n}{2}-x\|\to 0$;
\item 
$x$ is a  {\em G\^ateaux smooth point} (respectively, {\em Fr\'echet smooth point})  of $B_X$ if, 
 the following limit
holds $$\lim_{h\to0}\frac{\|x+hy\|+\|x-hy\|-2}{h}=0,$$
whenever $y\in S_X$ (respectively, uniformly for  $y\in S_X$).
\end{itemize}
\end{definition}

We say that a norm $\|\cdot\|$ on $X$ is LUR if each point $x\in S_X$ is a locally uniformly rotund point. Similarly one can define the notions of MLUR, WLUR, WMLUR, G\^ateaux smooth, Fr\'echet smooth norm. Finally, a norm $\|\cdot\|$ is said \textit{weakly uniformly rotund (WUR)} if $x_n-y_n\to 0$ in the weak topology of $X$ whenever $x_n,y_n\in S_X$ are such that $\|x_n + y_n\|\to 2$. 

\subsection{Some remarks on Draga's norm}\label{sec: draganorm} Let $(X,\|\cdot\|)$ be a separable Banach space and $(e_n,f_n)_{n\in\N}$ a bounded M-basis. Moreover, we assume that $\|e_n\|=1$, for every $n\in\N$. Define, for $x\in X$ 
\begin{equation*}
    \|x\|_0=\max\{\frac{1}{2}\|x\|,\sup_n|f_n(x)|\}
\end{equation*}
and
\begin{equation}\label{e: normDraga}
   |\!|\!|x|\!|\!|^2=\|x\|^2_0 + \sum_{n=1}^\infty 4^{-n}|f_n(x)|^2. 
\end{equation}
The equivalent norm defined in \eqref{e: normDraga} was introduced by S. Draga in \cite{D15} for proving that \textit{every infinite-dimensional Banach space with separable dual admits an equivalent WLUR norm which is not LUR}. It turns out that the same norm can be used for solving affirmatively the two open problems \cite[Problem 52.3.3]{GMZ} and \cite[Problem 52.3.6]{GMZ}. Indeed, by a standard argument, the norm $|\!|\!|\cdot|\!|\!|$ is rotund (see e.g. \cite[Proposition 161]{GMZ}). Let us show that the norm $|\!|\!|\cdot|\!|\!|$ is not MLUR. Let
\begin{equation*}
    x=e_1, \quad x_n=e_1+e_n, \quad y_n=e_1-e_n.
\end{equation*}
We observe that
\begin{equation*}
    |\!|\!|x|\!|\!|^2=\frac{5}{4},\quad |\!|\!|x_n|\!|\!|^2= \frac{5}{4} + \frac{1}{4^n}, \quad  |\!|\!|y_n|\!|\!|^2= \frac{5}{4} + \frac{1}{4^n}. 
\end{equation*}
Therefore,
\begin{equation*}
    |\!|\!|x_n|\!|\!|\to |\!|\!|x|\!|\!| \mbox{ and }  |\!|\!|y_n|\!|\!|\to |\!|\!|x|\!|\!|. 
\end{equation*}
Observing that
\begin{equation*}
   \textstyle{|\!|\!|}\frac{x_n + y_n}{2} - x{|\!|\!|}=0 
\end{equation*}
and that
\begin{equation*}
    |\!|\!|x_n-y_n|\!|\!|=2|\!|\!|e_n|\!|\!|=2 + \frac{2}{4^n}\geq 2, 
\end{equation*}
for every $n\in\N$, we get that $|\!|\!|\cdot|\!|\!|$ is not MLUR and \cite[Problem 52.3.3]{GMZ} is solved.

Finally, if additionally we assume that $X^*$ is separable, we can consider the M-basis $(e_n,f_n)_{n\in\N}$ being shrinking (i.e.
 $\overline{\operatorname{span}}\{f_n\colon n\in \N\}=X^*$). Thus, we get easily that the norm $|\!|\!|\cdot|\!|\!|$ is WUR and not MLUR (this solves \cite[Problem 52.3.6]{GMZ}).
On the other hand, it is standard to show that the norm $|\!|\!|\cdot|\!|\!|$ is not G\^ateaux smooth at $x=e_1+e_2$. Indeed, for instance it is enough to observe that
\begin{equation*}
    \lim_{t\to 0^-}\frac{|\!|\!|e_1+e_2+te_1|\!|\!|-|\!|\!|e_1+e_2|\!|\!|}{t} =\frac{\sqrt{21}}{21},
\end{equation*}
while
\begin{equation*}
    \liminf_{t\to 0^+} \frac{|\!|\!|e_1+e_2+te_1|\!|\!|-|\!|\!|e_1+e_2|\!|\!|}{t} \geq\frac{5\sqrt{21}}{21}.
\end{equation*}
In the third section, we provide, in every separable Banach space (Banach space with separable dual, respectively), an example of a rotund (WUR, respectively) not MLUR norm, which is additionally G\^ateaux smooth (Fr\'echet smooth, respectively), see Theorem \ref{th: mainG} and Theorem \ref{th: mainF} below.

\begin{remark}
    In \cite[Theorem 2.1]{YOST} D. Yost provided, in every separable Banach space, a rotund renorming which is not WLUR. It turns out, by the same argument as above, that such a norm is not MLUR. Even this norm is not G\^ateaux smooth.
\end{remark}

\subsection{A consequence of relatively compactness of almost overcomplete sequence} We conclude this preliminary section by showing that if a norm on a Banach space $X$ is not LUR, then there exists a closed infinite-codimensional subspace in which the restriction of the initial norm is not LUR. Although this result is not surprising, its proof is not trivial and requires some additional tools and moreover might be of independent interest. For these reasons we decided to include it in the manuscript. Notice that the argument used in Proposition \ref{p: applicationFZ} can be applied to other rotundity properties.  We recall that a sequence $\{x_n\}_{n\in\N}$ in a separable Banach space $X$ is said \textit{almost overcomplete} if $\overline{\operatorname{span}}\{x_{n_k}\}_{k\in\N}$ has finite codimension in $X$, for every subsequence $\{x_{n_k}\}_{k\in\N}$. By \cite[Theorem 2.1]{FonfZanco} if $\{x_n\}_{n\in\N}\subset X$ is almost overcomplete and bounded, then it is relatively compact (see also \cite[Corollary 5.5]{RS21} for an alternative proof). Finally, we refer to \cite{FonfZanco,FSTZ,Kos21,RS21} and references therein for further information on overcomplete sets. 

\begin{proposition}\label{p: applicationFZ}
    Let $(X,\|\cdot\|)$ be a separable Banach space.  Suppose that $\|\cdot\|$ is not LUR. Then there exists a closed infinite-codimensional subspace $Y\subset X$ such that $\|\cdot\|$ restricted to $Y$ is not LUR.
\end{proposition}

\begin{proof}
Since $(X,\|\cdot\|)$ is not LUR, there exists a point $x\in S_X$ and a sequence $\{x_n\}_{n\in\N}\subset S_X$ such that $\|x_n+x\|\to 2$ for $n\to \infty$, and $\{x_n\}_{n\in\N}$ does not converge to $x$. 

We may suppose, without loss of generality, that $\{x_n\}_{n\in\N}$ is almost overcomplete. Indeed, if not, we would get easily the assertion.

Since $\{x_n\}_{n\in\N}$ is almost overcomplete and bounded, by \cite[Theorem 2.1]{FonfZanco}, $\{x_n\}_{n\in\N}$ is relatively compact. Hence, we may assume that $x_n\to y$ for some $y\in S_X$.

Since $\{x_n\}_{n\in\N}$ does not converge to  $x$, the element $y$ is different from $x$. Furthermore we have
\begin{equation*}
   \bigg \|\frac{x_n+x}{2}\bigg\|\to \bigg\|\frac{x+y}{2}\bigg\| \mbox{ and } \bigg\|\frac{x_n+x}{2}\bigg\|\to 1.
\end{equation*}
Thus, the norm $\|\cdot\|$ restricted to $Y=\mathrm{span}\{x,y\}$ is not strictly convex.
\end{proof}

\section{Main results}\label{sec: main result}

The aim of the present section is to prove Theorem \ref{th: thm A}. We start by proving a technical lemma which ensures the existence, in every separable Banach space, of a smooth renorming which is not MLUR.

\begin{lemma}\label{lemma: flatnormingG}
    Let $(X,\|\cdot\|)$ be a separable  Banach space with rotund dual norm. Then, for each $\epsilon\in (0,1)$, there exist an equivalent norm $|\!|\!|\cdot|\!|\!|$, and a bounded M-basis   $(e_n, f_n)_{n\in\N}$ on $X$ such that:
    \begin{enumerate}
        \item $|\!|\!|\cdot|\!|\!|$  is G\^ateaux smooth;
        \item $|\!|\!|e_n|\!|\!|=1$, for every $n\in\N$;
\item for each $y\in \ker (f_1)$ with $\|y\|\leq \epsilon$, we have that $|\!|\!|e_1+y|\!|\!|=1$;
\item $|\!|\!|x|\!|\!|\leq \|x\|\leq{(1+\epsilon)}|\!|\!|x|\!|\!|$, for every $x\in X$.
    \end{enumerate}  
\end{lemma}

\begin{proof}
Let   $(x_n, x^*_n)_{n\in\N}$ be a bounded M-basis    on $X$ such that
$\|x^*_1\|^*=\|x_1\|=1$.
Define $S=S(B_{\|\cdot\|}, x^*_1,\epsilon)$ and let $\|\cdot\|_1$ be the equivalent norm on $X$ such that 
$$B_{\|\cdot\|_1}=B_{\|\cdot\|}\setminus[S\cup(-S)];$$
observe that:
\begin{enumerate}[(a)]
    \item $(1-\epsilon)B_{\|\cdot\|}\subset B_{\|\cdot\|_1}\subset B_{\|\cdot\|}$;
    \item since $(1-\epsilon)x_1\in U_{\|\cdot\|}$, for each $y\in \ker (x^*_1)$ with $\|y\|\leq \epsilon$, we have  $\|(1-\epsilon)x_1+y\|_1=1$. 
\end{enumerate}

Now, let  $|\!|\!|\cdot|\!|\!|$ be the equivalent norm defined by
\begin{equation}\label{eq: deftrestanghette}
B_{|\!|\!|\cdot|\!|\!|}=\overline{B_{\|\cdot\|_1}+\epsilon B_{\|\cdot\|}},
\end{equation}
and observe that, by \cite[Table~3.1, p.54]{GMZ}, we have 
$$|\!|\!|\cdot|\!|\!|^*=
\|\cdot\|^*_1+\epsilon \|\cdot\|^*.$$
Since by our assumption $\|\cdot\|^*$ is rotund, by \cite[Proposition~152]{GMZ},
$|\!|\!|\cdot|\!|\!|^*$ is rotund too. In particular (i) holds (see \cite[Theorem 143(i)]{GMZ}). Now, observe that, by definition of $|\!|\!|\cdot|\!|\!|$ and $\|\cdot\|_1$, we have 
\begin{equation}\label{eq: sup=1}
    \sup x_1^*(B_{|\!|\!|\cdot|\!|\!|})=\sup x_1^*(B_{\|\cdot\|_1})+\epsilon \sup x_1^*(B_{\|\cdot\|})=1-\epsilon+\epsilon=1.
\end{equation}
Since $x^*_1(x_1)=1$, combining \eqref{eq: deftrestanghette} and \eqref{eq: sup=1} we get $|\!|\!|x_1|\!|\!|=1$.  For each $n\in\N$, define $e_n=\frac{x_n}{|\!|\!|x_n|\!|\!|}$ and $f_n=|\!|\!|x_n|\!|\!|x_n^*$, clearly $(e_n,f_n)_{n\in\N}$ is a bounded M-basis and (ii) holds. 

Let us prove (iii). Take $y\in\ker(f_1)=\ker(x_1^*)$ with $\|y\|\leq\epsilon$ and observe that, by (b), we have
$$e_1+y=(1-\epsilon)x_1+y+\epsilon x_1\in B_{\|\cdot\|_1}+\epsilon B_{\|\cdot\|}\subset B_{|\!|\!|\cdot|\!|\!|};$$
moreover, $x_1^*(e_1+y)=1$. Since, by \eqref{eq: sup=1}, 
 $\sup x_1^*(B_{|\!|\!|\cdot|\!|\!|})=1$, we necessarily have $|\!|\!|e_1+y|\!|\!|=1$, and 
(iii) is proved.

Finally, (iv) follows immediately from the definitions of the balls.
\end{proof}

Next lemma is the Fr\'echet smooth counterpart of the previous one. We refer to the appendix for an alternative proof of Lemma \ref{lemma: flatnormingF} which relies on Proposition \ref{p: LUR sum}.

\begin{lemma}\label{lemma: flatnormingF}
    Let $(X,\|\cdot\|)$ be a separable  Banach space with LUR dual norm. Then, for each   $\epsilon\in(0,1)$, there exist $\gamma>0$, an equivalent norm $|\!|\!|\cdot|\!|\!|$, and a bounded M-basis   $(e_n, f_n)_{n\in\N}$ on $X$ such that:
    \begin{enumerate}
        \item $|\!|\!|\cdot|\!|\!|$  is Fr\'echet smooth;
        \item $|\!|\!|e_n|\!|\!|=1$, for every $n\in\N$;
\item for each $y\in \ker (f_1)$ with $\|y\|\leq \gamma$, we have that $|\!|\!|e_1+y|\!|\!|=1$;
\item $\sqrt{1-\epsilon}|\!|\!|x|\!|\!|\leq \|x\|\leq\sqrt{1+\epsilon}|\!|\!|x|\!|\!|$, for every $x\in X$.
    \end{enumerate}  
\end{lemma}

\begin{proof}
Let $(x_{n}, x^{*}_{n})_{n\in \mathbb{N}}$ be a bounded M-basis on
$X$ such that $\|x^{*}_{1}\|^{*}=\|x_{1}\|=1$. Let
$\|\cdot \|_{1}$ be the equivalent norm on $X$ such that
\begin{equation*}
B_{\|\cdot \|_{1}}=B_{\|\cdot \|}\setminus [S\cup (-S)],
\end{equation*}
where
$S=S(B_{\|\cdot \|}, x^{*}_{1},1-\sqrt{1-\varepsilon })$. Observe that:
\begin{enumerate}[(a)]
    \item $\sqrt{1-\epsilon}\,B_{\|\cdot\|}\subset B_{\|\cdot\|_1}\subset B_{\|\cdot\|}$;
    \item since $\sqrt{1-\epsilon}\,x_1\in U_{\|\cdot\|}$, there exists $\delta>0$ such that, for each $y\in \ker (x^*_1)$ with $\|y\|\leq \delta$, we have  $\|x_1+y\|_1=\|x_1\|_1$. 
\end{enumerate}

Now, let  $|\!|\!|\cdot|\!|\!|$ be the equivalent norm defined, for $x\in X$, by
$$\textstyle|\!|\!|x|\!|\!|^2=\inf\{{\|u\|_1^2+\frac{1}{\epsilon}\|v\|^2};\, u,v\in X, u+v=x\}.$$
Thanks to \cite[Proposition~5.2]{DESOVESTAR}, we have:
\begin{itemize}
	\item $|\!|\!|x|\!|\!|\leq\|x\|_1\leq \sqrt{1+\epsilon}\,|\!|\!|x|\!|\!|$;
	\item the dual norm of $|\!|\!|\cdot|\!|\!|$ is given by:
	$$\textstyle |\!|\!|x^*|\!|\!|^*=\sqrt{[\|x^*\|^*_1]^2+\epsilon[\|x^*\|^*]^2}\ \ \ \ \ \ \ (x^*\in X^*).$$
	\item $|\!|\!|\cdot|\!|\!|$ is Fr\'echet smooth (since its dual norm is LUR). 
\end{itemize}
Hence, (i) and (iv) hold. Moreover, if, for each $n\in\N$, define $e_n=\frac{x_n}{|\!|\!|x_n|\!|\!|}$ and $f_n=|\!|\!|x_n|\!|\!|x_n^*$, then clearly $(e_n,f_n)_{n\in\N}$ is a bounded M-basis and (ii) holds.  

It remains to prove (iii). We claim that if $x\in X$ is such that $x_1^*(x)=1$ then 
$|\!|\!|x|\!|\!|\geq 1$. Indeed, let $u,v\in X$ be such that $x=u+v$, if we denote $\theta=x_1^*(u)$, since $\|x_1^*\|_1^*=\sup x_1^*(B_{\|\cdot\|_1})=\sqrt{1-\epsilon}$ and $\|x_1^*\|^*=\sup x_1^*(B_{\|\cdot\|})=1$, we clearly have:
$$\textstyle x_1^*(v)=1-\theta,\quad \|u\|_1\geq\frac{1}{\sqrt{1-\epsilon}}|\theta|,\quad  \|v\|\geq |1-\theta|.$$
Hence,
\begin{eqnarray*}
    \textstyle |\!|\!|x|\!|\!|^2&\geq& \textstyle  
\inf\{\frac{1}{(1-\epsilon)}\theta^2+\frac{1}{\epsilon}(1-\theta)^2;\, u,v\in X, u+v=x\}\\
&\geq&  \textstyle 
\min_{\theta\in\R}\left[\frac{1}{(1-\epsilon)}\theta^2+\frac{1}{\epsilon}(1-\theta)^2\right]\\
&=&  \textstyle\frac{1}{(1-\epsilon)}(1-\epsilon)^2+\frac{1}{\epsilon}(1-(1-\epsilon))^2=1,
\end{eqnarray*}
and the claim is proved.

In particular, we have $|\!|\!|x_1|\!|\!|\geq 1$.
Moreover, since we can write 
$$\textstyle x_1=\theta_{\min} x_1+(1-\theta_{\min}) x_1,$$
where $\theta_{\min}=1-\epsilon$, we have $|\!|\!|x_1|\!|\!|= 1$, and hence $x_1=e_1$.

Now, observe that if  $y\in \ker (f_1)=\ker(x_1^*)$ then $x_1^*(x_1+y)=1$ and hence, by our claim,
$|\!|\!|x_1+y|\!|\!|\geq 1$. Moreover, by (b), if $\gamma=\theta_{\min}\delta$,  for each $y\in \ker (x^*_1)$ with $\|y\|\leq \gamma$, we have  $\|\theta_{\min}x_1+y\|_1=\|\theta_{\min}x_1\|_1$.
Since, for each $y\in \ker (x^*_1)$ with $\|y\|\leq \gamma$, we can write 
$$\textstyle e_1+y=x_1+y=\theta_{\min} x_1+y+(1-\theta_{\min})x_1,$$
we have 
\begin{eqnarray*}
    \textstyle |\!|\!|e_1+y|\!|\!|^2=|\!|\!|x_1+y|\!|\!|^2&=& \textstyle  \inf\{{\|u\|_1^2+\frac{1}{\epsilon}\|v\|^2};\, u,v\in X, u+v=x\}\\
 &\leq&\textstyle \|\theta_{\min}x_1+y \|_1^2+\frac{1}{\epsilon}\|(1-\theta_{\min})x_1\|^2\\ 
&=&\textstyle  \frac{1}{(1-\epsilon)}(1-\epsilon)^2+\frac{1}{\epsilon}(1-(1-\epsilon))^2=1.
\end{eqnarray*}
Hence, $|\!|\!|e_1+y|\!|\!|=1$ whenever $y\in \ker (f_1)$ with $\|y\|\leq \gamma$, and the proof is concluded.
\end{proof}

We are in the position for proving the main results of the section.

\begin{theorem}\label{th: mainG}
 Let $(X,\|\cdot\|)$ be a separable Banach space. Then there exists a
 rotund G\^ateaux  smooth  equivalent norm  which is not midpoint locally uniformly rotund. Moreover, such a norm can be chosen arbitrarily close to $\|\cdot\|$.
\end{theorem}

\begin{proof}
By \cite[Theorem~4.1 in \S II]{DGZ} and \cite[Theorem~2.6(i) in \S II]{DGZ} there exists an equivalent norm on $X$ which is arbitrarily close to the original norm $\|\cdot\|$ and its dual norm is rotund. Therefore, we may suppose without loss of generality that $(X,\|\cdot\|)$ has a rotund dual norm. Let $\varepsilon>0$. By Lemma \ref{lemma: flatnormingG} there exist an equivalent norm $|\!|\!|\cdot|\!|\!|$ and a bounded M-basis $(e_n,f_n)_{n\in\mathbb{N}}$ on $X$ such that (i)-(iv) hold. We define the equivalent norm 
\begin{equation}\label{e: RGnoMLUR}
|x|^2=|\!|\!|x|\!|\!|^2 + \sum_{n=1}^\infty (4^{-n}|f_n(x)|^2)
\end{equation}
on $X$. Since $(e_n,f_n)_{n\in\N}$ is a bounded M-basis the norm $|\cdot|$ is well-defined and rotund. G\^ateaux smoothness of the norm in \eqref{e: RGnoMLUR} follows by Lemma \ref{lemma: flatnormingG} (i). It remains to prove that  $|\cdot|$ is not MLUR. In order to do so, let 
\begin{equation*}
    x=e_1 \quad x_n=e_1+\frac{\epsilon}{1+\varepsilon}e_n \quad y_n=e_1-\frac{\epsilon}{1+\varepsilon}e_n.
\end{equation*}
We observe that, by Lemma \ref{lemma: flatnormingG} (ii), we have 
\begin{equation*}
    |x|^2=|\!|\!|e_1|\!|\!|^2 + \frac{1}{4} = \frac{5}{4}.
\end{equation*}
By Lemma \ref{lemma: flatnormingG} (ii) and (iv), it follows that
\begin{equation*}
  \bigg \| \frac{\epsilon}{1+\varepsilon}e_n\bigg\|\leq \epsilon |\!|\!|e_n|\!|\!|=\epsilon.
\end{equation*}
Therefore, by Lemma \ref{lemma: flatnormingG} (iii), we obtain
\begin{equation*}
    |x_n|^2= |\!|\!|e_1+\frac{\epsilon}{1+\varepsilon}e_n|\!|\!|^2 + \frac{1}{4} + \frac{1}{4^n}\frac{\epsilon^2}{(1+\epsilon)^2}=\frac{5}{4} + \frac{1}{4^n}\frac{\epsilon^2}{(1+\epsilon)^2},
\end{equation*}
for every $n\in \N$. By a similar argument we get 
\begin{equation*}
    |y_n|^2= |\!|\!|e_1-\frac{\epsilon}{1+\varepsilon}e_n|\!|\!|^2 + \frac{1}{4} + \frac{1}{4^n}\frac{\epsilon^2}{(1+\epsilon)^2}=\frac{5}{4} + \frac{1}{4^n}\frac{\epsilon^2}{(1+\epsilon)^2},
\end{equation*}
for every $n\in \N$. Therefore, we have
\begin{equation*}
    |x_n|\to |x|\quad \mbox{ and }\quad |y_n|\to |x|.
\end{equation*}
We observe, for every $n\in \N$, that
\begin{equation*}
    \bigg|\frac{x_n+y_n}{2} - x\bigg|=0
\end{equation*}
and 
\begin{equation*}
    |x_n-y_n|=\frac{2\epsilon}{1+\varepsilon}|e_n|=\frac{2\epsilon}{1+\varepsilon}\bigg(1+\frac{1}{4^n}\bigg)^{\frac{1}{2}}\geq \frac{2\epsilon}{1+\varepsilon}
\end{equation*}
Thus, the norm $|\cdot|$ is not MLUR at $x=e_1$.\\
The density follows by Lemma \ref{lemma: flatnormingG} (iv) and considering the equivalent norm
\begin{equation*}
|x|^2_{\varepsilon}=|\!|\!|x|\!|\!|^2 + \varepsilon\sum_{n=1}^\infty (4^{-n}|f_n(x)|^2)
\end{equation*}
on $X$.
\end{proof}

\begin{theorem}\label{th: mainF}
 Let $(X,\|\cdot\|)$ be a  Banach space with separable dual. Then there exists a
 WUR Fr\'echet smooth  equivalent norm  which is not midpoint locally uniformly rotund. Moreover, such a norm can be chosen arbitrarily close to $\|\cdot\|$.
\end{theorem}

\begin{proof}
By \cite[Theorem~4.1(ii) in \S II]{DGZ} and \cite[Theorem~2.6(ii) in \S II]{DGZ} there exists an equivalent norm on $X$ which is arbitrarily close to the original norm $\|\cdot\|$ and its dual norm is LUR.  Then, the proof follows in the same way as the previous theorem by using Lemma \ref{lemma: flatnormingF} (or equivalently Lemma \ref{lemma flatnormindFalternative}) instead of Lemma \ref{lemma: flatnormingG}.
\end{proof}

\begin{remark}\label{r: URED} 
    In \cite{HQ} P. H\'ajek and A. Quilis have constructed in every infinite-dimensional separable Banach space a norm which is LUR but not URED (uniformly rotund in every direction), solving positively \cite[Problem 52.3.4]{GMZ}. On the other hand,
    it is worth to mention that by a standard application of \cite[Exercise 9.20]{FHHMZ} the norm defined in Theorem \ref{th: mainG} as well as the norm defined by \eqref{e: normDraga} is URED. Therefore, both norms provide an affirmative answer to \cite[Problem 52.3.5]{GMZ} which states: \textit{Can every infinite-dimensional separable space be renormed to be URED but not LUR?}. 
\end{remark}

\begin{remark}
    In \cite{DEBESOM-ALUR} the first and the third named authors provide, in every infinite-dimensional separable Banach space, an example of rotund G\^ateaux smooth norm which satisfies the Kadec property but it is not LUR. This construction provide a positive answer to \cite[Problem 52.1.1.5]{GMZ}. Observe that Theorem \ref{th: mainG} provides an alternative solution to the same problem. On the other hand, the norm defined in \cite{DEBESOM-ALUR} does not solve 
    \cite[Problem 52.1.1.3]{GMZ}. Indeed,  if a norm is rotund and satisfies the Kadec property then it is MLUR (see \cite[Corollary 408]{GMZ}).
\end{remark}

The next question have been asked to us by Tommaso Russo and to the best of our knowledge it seems to be open.

\begin{problem}
    Can every infinite-dimensional separable Banach space be re\-normed to be rotund [and  G\^ateaux  smooth]  but with no LUR point in the unit sphere?
\end{problem}

In this direction, it is worth to notice that in \cite{Mo83} P. Morris proved that any separable Banach space containing an isomorphic copy of $c_0$ admits a rotund equivalent norm such that no point of the unit sphere is WMLUR. By using a result in \cite{Ha95} it is possible to show that if $X$ is a polyhedral space, then it admits a $C^{\infty}$ rotund equivalent norm such that no point of the unit sphere is WMLUR (see \cite[Theorem~2.1]{GMZ15}). Finally, in \cite{GMZ15} it is proved that every $C^2$ norm on a polyhedral space admits no MLUR point in the unit sphere. In particular it exists, in every separable polyhedral space, an equivalen WUR, $C^2$ norm with no MLUR points (see \cite[Theorem~2.4]{GMZ15}). We continue by showing that it is possible to define, in each separable Banach space, an equivalent rotund norm in which the set of MLUR points is porous in the unit sphere. Let us start by recalling the definition of porous set.

\begin{definition}
    A subset $S$ of a metric space $(T,d)$ is called {\em porous} if there are numbers
$\lambda\in(0,1)$ and $\delta_0>0$ such that for every $x\in S$  and every $0<\delta<\delta_0$ there exists $y\in T$  such that
$d(x,y)<\delta$ and $S\cap \{t\in T;\, d(t,y)<\lambda\delta\}=\emptyset$.
\end{definition}

\begin{theorem}\label{th: MLUR porous}
Each infinite-dimensional separable Banach space admits an equivalent rotund norm such that the set of  MLUR points is porous in the unit sphere. 
\end{theorem}

\begin{proof}
By a result due to R. Partington (see \cite[Theorem~1]{P}), there exist $\lambda\in (0,1)$, an equivalent norm $\|\cdot\|$ on $X$, and
sequences $\{e_n\}_{n\in\N}\subset X$, $\{f_n\}_{n\in\N}\subset X^*$ such that:
\begin{enumerate}[(a)]
\item $f_n(e_n)=\|f_n\|^*=1$, whenever $n\in\N$;
\item $\|x\|=\sup_n |f_n(x)|$, whenever $x\in X$;
\item $|f_n(e_k)|<\lambda$, whenever $n,k\in\N$ and $k\neq n$.
\end{enumerate}
Let us denote, for $n\in \N$,
$$H_n=\{x\in X; \, f_n(x)=1\},\qquad S_n=S_{(X,\|\cdot\|)}\cap H_n,\qquad D=\pm\bigcup_n \inte_{H_n} (S_n).$$
We claim that $P:=S_{(X,\|\cdot\|)}\setminus D$ is a porous set in $S_{(X,\|\cdot\|)}$.
Let $x\in S_{(X,\|\cdot\|)}$  and  let $\delta\in(0,1)$. Then, proceeding as in the proof of \cite[Proposition 3.2]{M}, it is easy to see that there exist $n\in\N$ and $z\in S_n$ such that the following conditions hold:
\begin{enumerate}
    \item  $\|z-x\|< \delta$;
    \item if we define $w=\delta e_n+(1-\delta)z$ then  $$\textstyle \left(w+\delta \frac{1-\lambda}{2}B_{(X,\|\cdot\|)}\right)\cap S_{(X,\|\cdot\|)}\subset S_n\qquad\text{and}\qquad \|w-z\|\leq 2\delta.$$
\end{enumerate}
In particular, we have proved that, for each $x\in S_{(X,\|\cdot\|)}$ and $\delta\in(0,1)$, there exists $w\in S_{(X,\|\cdot\|)}$ such that $\|w-x\|\leq3\delta$ and such that
$$\textstyle\left\{t\in S_{(X,\|\cdot\|)};\, \|t-w\|<\delta \frac{1-\lambda}{2}\right\}\subset  D.$$
Our claim is proved, since the constant $\frac{1-\lambda}{2}$ depends just on the norm $\|\cdot\|$.

 Now, consider the equivalent  norm $\|\cdot\|_1$  on $X$ defined by
\begin{equation*}
\|x\|_1^2=\|x\|^2 + \sum_{n=1}^\infty 4^{-n}|f_n(x)|^2,\qquad\qquad x\in X.
\end{equation*}
Observe that, since $\{f_n\}_{n\in\N}$ is a bounded separating family for $X$, $\|\cdot\|_1$ is an equivalent rotund norm on $X$.

We claim that the set of all  MLUR points for $B_{(X,\|\cdot\|_1)}$ is a porous set in the unit sphere $S_{(X,\|\cdot\|_1)}$. In order to prove our claim, since the map $x\mapsto \frac{x}{\|x\|_1}$ is a bi-Lipschitz function between $S_{(X,\|\cdot\|)}$ and $S_{(X,\|\cdot\|_1)}$,
 it is sufficient to prove that if $x\in D$ then $\frac{x}{\|x\|_1}$ is not an MLUR point for $B_{(X,\|\cdot\|_1)}$. Let us fix $x\in D$. Without any loss of generality, we can suppose that there exist $n_0\in\N$ and $\epsilon\in(0,1)$ such that
$$\left(x+\epsilon B_{(X,\|\cdot\|)}\right)\cap H_{n_0}\subset S_{n_0}.$$
For each $N\in\N$, consider an element $v_N$ belonging to the set
$$\epsilon S_{(X,\|\cdot\|)}\cap\ker(f_{n_0})\cap\ker(f_1)\cap\ldots\cap\ker(f_N).$$
Then we clearly have that $x\pm v_N\in S_{n_0}\subset S_{(X,\|\cdot\|)}$  and that:
\begin{eqnarray*}
\|x\pm v_N\|_1^2 &=& \|x\pm v_N\|^2 + \sum_{n=1}^\infty 4^{-n}|f_n(x\pm v_N)|^2\\
&=& 1 + \sum_{n=1}^N 4^{-n}|f_n(x)|^2+ \sum_{n=N+1}^\infty 4^{-n}|f_n(x\pm v_N)|^2.
\end{eqnarray*}
Hence,
\begin{eqnarray*}
\left|\|x\pm v_N\|_1^2-\|x\|_1^2\right| &=&  \left|\sum_{n=N+1}^\infty 4^{-n}|f_n(x\pm v_N)|^2-\sum_{n=N+1}^\infty 4^{-n}|f_n(x)|^2\right|\\
&\leq& 2(1+\epsilon)^2\sum_{n=N+1}^\infty 4^{-n}.
\end{eqnarray*}
In particular, we have that $\|x+ v_N\|_1\to \|x\|_1$ and $\|x- v_N\|_1\to \|x\|_1$, as $N\to\infty$. From the fact that $v_N\not\to 0$ {(notice that $\|v_N\|=\varepsilon$ for each $N\in\N$)}, as $N\to\infty$, it follows that $\frac{x}{\|x\|_1}$ is not an MLUR point for $B_{(X,\|\cdot\|_1)}$, and our claim is proved.
\end{proof}

Notice that in Theorem~\ref{th: MLUR porous} we cannot replace the condition  that the set of MLUR points is porous with the condition that the set of MLUR points is empty. Indeed, if $X$ is a Banach space with the Radon-Nikodym property (RNP) then every renorming of $X$ is such that the corresponding unit ball admits strongly exposed points (see, e.g.,  \cite[Theorem~3.5.4]{Bourgin}) and hence it admits  in particular MLUR points (cf. the introduction of  \cite{Mo83}). As a corollary we have, for example, that 
the Banach space $\ell_1$, having the RNP, does not admit any (rotund) renorming without MLUR points.
Observe that for LUR points the situation is different, the following example, inspired by the recent paper \cite{CH}, shows that $\ell_1$  admits an equivalent  G\^ateaux  smooth rotund norm $\|\cdot\|$ such that the corresponding unit sphere does not have LUR points. In order to do that, we need to recall the definition of octahedral norm.

\begin{definition}
    A Banach space $(X,\|\cdot\|)$ (or its norm) is \textit{octahedral} if, for every finite dimensional subspace $F$ of $X$ and every $\epsilon>0$, there is a $y\in S_X$ such that
    \begin{equation}\label{e: octa}
          \|x+\lambda y\|\geq (1-\epsilon)(\|x\| + |\lambda|)
    \end{equation}
for every $x\in F$ and every $\lambda\in \R$.
\end{definition}

It is standard to show that in the definition above one can equivalently require that \eqref{e: octa} holds for every $x\in S_F$ and every $\lambda\in \R$ (see for example \cite[Proposition 2.2]{HallerLangementsPoldvere}). 
The following result shows that the unit sphere of an octahedral space does not admit LUR points. Although the proof of this result is elementary, we decided to include it because we did not find it in the literature.

\begin{proposition}\label{p: octnoLUR}
The unit sphere of an octahedral   Banach space $X$ does not admit LUR points.
\end{proposition}

\begin{proof} 
     Suppose on the contrary  that  $x_0\in S_X$ is an LUR point of $B_X$. Let $F=\operatorname{span}(x_0)$. For each $n\in\N$, by definition of octahedral norm, there exists a point $y_n\in S_X$ such that $$\|x+y_n\|\geq (1-\frac1n)(\|x\|+1),$$ 
for every $x\in F$. Therefore both $\|x_0+y_n\|$ and $\|-x_0+y_n\|$ go to $2$ as $n\to +\infty$. Thus, since $y_n$ cannot converge both to $x_0$ and $-x_0$, we get a contradiction.
\end{proof}

Finally, we recall that a Banach space $X$ admits an octahedral renorming if and only if $X$ contains an isomorphic copy of $\ell_1$ (see \cite[Theorem 2.5 in \textsection III]{DGZ}).

\medskip
In \cite{CH} C. Cobollo and P. H\'ajek define an equivalent norm on $\ell_1$, in fact in every Banach space admitting a G\^ateaux smooth renorming and containing a complemented copy of $\ell_1$, which is simultaneously G\^ateaux smooth and octahedral. We slightly modify their norm in order to obtain an equivalent norm on $\ell_1$ which is rotund, G\^ateaux smooth and octahedral (see Example \ref{e: octG} below). Therefore, applying Proposition \ref{p: octnoLUR} we get an equivalent norm on $\ell_1$ which is rotund, G\^ateaux smooth, and whose unit sphere has no LUR points. We denote by $|\!|\!|\cdot|\!|\!|$ the norm introduced in \cite{CH}, in the particular case in which the space considered is exactly $\ell_1$. We recall here only the necessary definitions and properties of this norm and refer to the original paper for all the details.
\begin{itemize}
    \item $P_n\colon\ell_1\to \ell_1$ is the projection onto $\operatorname{span}\{e_1,\dots,e_n\}$, where $\{e_n\}_{n\in\N}$ is the usual Schauder basis of $\ell_1$.
    \item $|\!|\!|x|\!|\!|=\sup_{n\in\N}\{|\!|\!|P_n x|\!|\!|_n\}$ for $x\in\ell_1$, where $|\!|\!|\cdot|\!|\!|_n$ is a suitable renorming of $X_n=P_n[\ell_1]$.
    \item $|\!|\!|e_n|\!|\!|=1$ for every $n\in\N$.
    \item Let $\epsilon>0$. Then there exists $n_0\in\N$ such that for every $n\geq n_0$
    \[
    |\!|\!|P_{n-1}x+e_n\alpha|\!|\!|\geq (1-\epsilon)(|\!|\!|P_{n-1}x|\!|\!|+|\alpha|),\qquad x\in\ell_1.
    \]
\end{itemize}
Where the last property holds combining \cite[Lemma 2.2]{CH}, \cite[Corollary 2.3]{CH}, and \cite[Proposition 2.6]{CH}.
We are in the position of defining the equivalent norm of $\ell_1$.
\begin{example}\label{e: octG} 
Let $|\!|\!|\cdot|\!|\!|$ be the equivalent norm of $\ell_1$ defined above. Let 
\[
 \|x\|=|\!|\!|x|\!|\!| + \sqrt{\sum_{n=1}^{+\infty}\frac{|x_n|^2}{4^n}}.
\]
The norm $\|\cdot\|$ is obviously rotund and G\^ateaux smooth. It remains to show that  $S_{(\ell_1,\|\cdot\|)}$ has no LUR points. In order to do that, in light of Proposition~\ref{p: octnoLUR}, it is enough to show that the space $(\ell_1,\|\cdot\|)$ is octahedral. Let $F\subset \ell_1$ be a finite dimensional subspace and $\epsilon>0$. Let $\delta>0$ be such that $\eta(\delta):=1-\delta^3+3\delta^2-4\delta\geq 1-\epsilon$. Such a $\delta>0$ exists since $\eta(\delta)\nearrow 1$ as $\delta\searrow 0$. There exists $n\in \N$ such that
\begin{itemize}
\item $\|x-P_n x\|\leq \delta$, for every $x\in S_{(F,\|\cdot\|)}$. Notice that by \cite{FHHMZ}*{Corollary~3.87} the sequence $\{P_n\}_{n\in\N}$ converges uniformly on compact sets to the identity map on $\ell_1$;
\item $|\!|\!|P_{n}x+e_{n+1}\beta|\!|\!|\geq (1-\delta)(|\!|\!|P_{n}x|\!|\!|+|\beta|)$ for every $x\in \ell_1$ and $\beta \in \R$;
\item $\frac{1}{1+2^{-n-1}}\geq 1-\delta$.
\end{itemize}
Then, for $x\in S_{(F,\|\cdot\|)}$ and $\alpha\in\R$, we have
\begin{equation*}
    \begin{split}
        \|x+\frac{\alpha}{1+2^{-n-1}}e_{n+1}\|&\geq \|P_n x +\frac{\alpha}{1+2^{-n-1}}e_{n+1}\|-\|x - P_n x\|\\
        &\geq |\!|\!|P_n x +\frac{\alpha}{1+2^{-n-1}}e_{n+1}|\!|\!| + \sqrt{\sum_{k=1}^{n}\frac{|x_k|^2}{4^k}} -\delta\\
        &\geq (1-\delta) (|\!|\!|P_n x|\!|\!|+ \frac{|\alpha|}{1+2^{-n-1}}) + \sqrt{\sum_{k=1}^{n}\frac{|x_k|^2}{4^k}} -\delta\\
      &\geq (1-\delta)^2(\|P_nx\| +|\alpha|) -\delta\\
        &\geq (1-\delta)^2(\|x\| +|\alpha| - \delta) -\delta\\
        &\geq (1-\delta)^2(\|x\| +|\alpha| - \delta(\|x\| +|\alpha|)) -\delta(\|x\| +|\alpha|)\\
        &=(1-\delta^3+3\delta^2-4\delta)(\|x\| +|\alpha|)\\
        &\geq (1-\varepsilon)(\|x\| +|\alpha|).
    \end{split}
\end{equation*}
Since $\left\|\frac{e_{n}}{1+2^{-n}}\right\|=1$ ($n\in\N$), we have that the space $(\ell_1,\|\cdot\|)$ is octahedral.
\end{example}

\section{Appendix}\label{sec: appendix}

We provide an alternative proof of Lemma \ref{lemma: flatnormingF}. Notice that the following statement is slightly different from the one in Lemma \ref{lemma: flatnormingF}. We will omit the full detailed proof, since it is enough to repeat the argument
 of Lemma \ref{lemma: flatnormingG} and applying Proposition \ref{p: LUR sum} below instead of \cite[Proposition 152]{GMZ} and applying \cite[Theorem 143(iii)]{GMZ} instead of \cite[Theorem 143(i)]{GMZ}.

\begin{lemma}\label{lemma flatnormindFalternative}
    Let $(X,\|\cdot\|)$ be a separable Banach space with LUR dual norm. Then, for each   $\epsilon\in(0,1)$, there exist an equivalent norm $|\!|\!|\cdot|\!|\!|$, and a bounded M-basis   $(e_n, f_n)_{n\in\N}$ on $X$ such that:
    \begin{enumerate}
        \item $|\!|\!|\cdot|\!|\!|$  is Fr\'echet smooth;
        \item $|\!|\!|e_n|\!|\!|=1$, for every $n\in\N$;
\item for each $y\in \ker (f_1)$ with $\|y\|\leq \epsilon$, we have that $|\!|\!|e_1+y|\!|\!|=1$;
\item $|\!|\!|x|\!|\!|\leq \|x\|\leq (1+\epsilon)|\!|\!|x|\!|\!|$, for every $x\in X$.
    \end{enumerate}  
\end{lemma}

The rest of the section is devoted to the key ingredient for proving Lemma \ref{lemma flatnormindFalternative}.

\subsection{Sum of equivalent LUR norms} 
Let $(X,\|\cdot\|_1)$ be a Banach space and $\|\cdot\|_2$ be a norm on $X$ such that $\|x\|_2\leq C\|x\|_1$, for some $C>0$ and every $x\in X$. Suppose that the norm $\|\cdot\|_2$ is rotund, then by a classical result (see e.g. \cite[Proposition 152]{GMZ}), the equivalent norm $\|\cdot\|=\|\cdot\|_1 + \|\cdot\|_2$ on $X$ is rotund too. On the other hand,  if we suppose that the norm $\|\cdot\|_2$ is LUR, it is not always true that the equivalent norm $\|\cdot\|=\|\cdot\|_1 + \|\cdot\|_2$ on $X$ is LUR, see for example \cite[Section 9.4.2]{GMZ}. The aim of this section is to show that if we assume that the norm $\|\cdot\|_2$ is LUR and additionally is equivalent to $\|\cdot\|_1$ then the norm $\|\cdot\|=\|\cdot\|_1+\|\cdot\|_2$ is LUR too.

The following fact is well-known (see, e.g., the proof of \cite[Fact~7.7]{FHHMZ}), however for the sake of completeness we include a proof.

\begin{fact}\label{fact: normalize}
    Let $(X,|\cdot|)$ be a normed space and $x,y\in X\setminus\{0\}$. If $|y|\leq|x|$ then we have
    $$\textstyle 2\geq\left|\frac{x}{|x|}+\frac{y}{|y|}\right|\geq 2-\frac{1}{|y|}(|x|+|y|-|x+y|).$$
\end{fact}
\begin{proof} The first inequality is trivial, for the latter it is sufficient to observe that 
$$\left|\frac{x}{|x|}+\frac{y}{|y|}\right|\geq \left|\frac{y}{|y|}+\frac{x}{|y|}\right|-\left|\frac{x}{|y|}-\frac{x}{|x|}\right|= 2-\frac{1}{|y|}(|x|+|y|-|x+y|)$$
\end{proof}

We shall also need the following elementary fact, whose proof is left to the reader.

\begin{fact} \label{fact: normalize2}
    Let $(X,\|\cdot\|)$ be a normed space, $0<a<b$, and $\{t_n\}\subset [a,b]$. Let $x_n,x\in S_X$ ($n\in\N$) be such that $\|x_n-t_n x\|\to 0$. Then $x_n\to x$ in norm.   
\end{fact}

\begin{proposition}\label{p: LUR sum}
    Let $(X,\|\cdot\|_1)$ be a normed space and let $\|\cdot\|_2$ be an equivalent LUR norm on $X$. Then the equivalent norm  on $X$ defined by $$\|\cdot\|=\|\cdot\|_1+\|\cdot\|_2$$ is LUR.
\end{proposition}

\begin{proof}
Suppose that  $x,x_n\in S_{(X,\|\cdot\|)}$ ($n\in\N$) are such that  $\|x+x_n\|\to2$. Let us prove that $x_n\to x$ in norm.  First, observe that 
$$\|x\|_1+\|x_n\|_1-\|x+x_n\|_1+\|x\|_2+\|x_n\|_2-\|x+x_n\|_2=2-\|x+x_n\|\to0.$$
In particular, since $\|x\|_i+\|x_n\|_i-\|x+x_n\|_i\geq0$ ($i=1,2$, $n\in\N$), we have that  $\|x\|_2+\|x_n\|_2-\|x+x_n\|_2\to0$. 

Since $\|\cdot\|_2$ is an equivalent norm on $X$, there exists  $M\geq 1$ such that $\|z\|\leq M\|z\|_2$, whenever $z\in X$.
We claim that, for $n\in\N$, we have  
\begin{equation}\label{eq: dis}
        \textstyle 2\geq\left\|\frac{x}{\|x\|_2}+\frac{x_n}{\|x_n\|_2}\right\|_2 \geq  2-M(\|x\|_2+\|x_n\|_2-\|x+x_n\|_2).
    \end{equation}
 To prove our claim, let us suppose without any loss of generality that $\|x_n\|_2\leq\|x\|_2$ (the case in which $\|x_n\|_2\geq\|x\|_2$ is similar).  
 By Fact~\ref{fact: normalize}, we have
$$\textstyle 2\geq\left\|\frac{x}{\|x\|_2}+\frac{x_n}{\|x_n\|_2}\right\|_2\geq 2-\frac{1}{\|x_n\|_2}(\|x\|_2+\|x_n\|_2-\|x+x_n\|_2).$$
Since $\frac{1}{\|x_n\|_2}\leq M\frac{1}{\|x_n\|}=M$, the claim follows.

Now, since  $\|x\|_2+\|x_n\|_2-\|x+x_n\|_2\to0$ and taking into account \eqref{eq: dis}, we have that $\left\|\frac{x}{\|x\|_2}+\frac{x_n}{\|x_n\|_2}\right\|_2\to2$. Since $\|\cdot\|_2$ is LUR, we have $\frac{x_n}{\|x_n\|_2}\to \frac{x}{\|x\|_2}$.
Observing that
$$\textstyle \left\| \frac{x_n}{\|x_n\|_2}- \frac{x}{\|x\|_2} \right\|=\frac{1}{\|x_n\|_2}\left\| x_n- \frac{\|x_n\|_2}{\|x\|_2}x \right\|\geq \left\| x_n- \frac{\|x_n\|_2}{\|x\|_2}x \right\|$$
and that $t_n:=\frac{\|x_n\|_2}{\|x\|_2}\in \left[\frac1M,M\right]$, whenever $n\in\N$,  Fact~\ref{fact: normalize2} implies that $x_n\to x$ in norm. The 
proof is concluded.
\end{proof}

\subsection*{Acknowledgements} We wish to thank the anonymous referee for carefully reading our manuscript and suggesting fruitful improvements.

\end{document}